\documentclass[letterpaper, 10 pt, conference]{ieeeconf}
\IEEEoverridecommandlockouts
\usepackage{relsize}
\usepackage{amsmath,bbm}
\usepackage{amsfonts}
\usepackage{comment} 
\usepackage{amssymb}
\usepackage{mathtools}

\DeclarePairedDelimiter\floor{\lfloor}{\rfloor}
\usepackage{tikz}
\usepackage{pgf, hyperref, cite}
\usepackage{subfig}
\usepackage{relsize}
\usetikzlibrary{arrows,automata}
\usepackage{algorithm}
\usepackage[noend]{algpseudocode}
\usepackage{dsfont}
\usepackage{graphicx}

\newtheorem{remark}{Remark}
\newtheorem{assumption}{Assumption}
\newcommand{\be}{\begin{equation}}
\newcommand{\ee}{\end{equation}}
\newcommand{\req}[1]{(\ref{#1})}
\newcommand{\argmin}{\mathop{\rm argmin}}

\usepackage{setspace}

\IEEEoverridecommandlockouts                              % This command is only
\def\ew#1{{{\color{black}#1}}}
\def\fm#1{{{\color{black}#1}}}
\def\ss#1{{{\color{black}#1}}}
\def\sm#1{{{\color{black}#1}}}

\newtheorem{theorem}{Theorem}[section]
\newtheorem{lemma}[theorem]{Lemma}

\newcommand{\norm}[1]{\left|\left|#1\right|\right|}

\title{A General Framework of Exact Primal-Dual First Order Algorithms for Distributed Optimization$^*$
\thanks{$^*$ Support provided by DARPA Langrange HR-001117S0039} }
\author{Fatemeh Mansoori$^\dag$\thanks{$^\dag$Department of Electrical and Computer Engineering, Northwestern University} and Ermin Wei$^\dag$}
\begin{document}
\maketitle
\begin{abstract}
 We study the problem of minimizing a sum of local objective convex functions over a network of processors/agents. This problem naturally calls for distributed optimization algorithms, in which the agents cooperatively solve the problem through local computations and communications with neighbors. While many of the existing distributed algorithms with constant stepsize can only converge to a neighborhood of optimal solution, some recent methods based on augmented Lagrangian and method of multipliers can achieve exact convergence with a fixed stepsize.  However, these methods either suffer from slow convergence speed or require minimization at each iteration. In this work, we develop a class of distributed first order primal-dual methods, which allows for multiple primal steps per iteration.  This general framework makes it possible to control the trade off between the performance and the execution complexity in primal-dual algorithms. We show that for strongly convex and Lipschitz gradient objective functions, this class of algorithms converges linearly to the optimal solution under appropriate constant stepsize choices. Simulation results confirm the superior performance of our algorithm compared to existing methods.  

\end{abstract}

\section{Introduction}
\label{intro}
\sm{This paper focuses on solving the following optimization problem
\be \tilde{x}^*:=\argmin_{\tilde{x}\in\mathbb{R}} \sum_{i=1}^n f_i(\tilde{x})\label{mainproblem}\ee
over a network of $n$ agents (processors)\footnote{For representation simplicity, we consider the case with $\tilde{x}\in\mathbb{R}$. Our analysis can be easily generalized to the multidimensional case.}. The agents are connected with an undirected static graph $\mathcal{G(V,E)}$, with $\mathcal{V}$ and $\mathcal{E}$ being the set of vertices and edges respectively. Set $\mathcal{N}_i$ denotes the  neighbors of agent $i$. Each agent $i$ in the network has access to a real-valued convex local objective function $f_i$, and can only communicate with its neighbors defined by the graph.} \par \sm{Problems of form \req{mainproblem} that require distributed optimization algorithms are applicable to many important  areas such as sensor networks,  smart grids, robotics, and machine learning \cite{predd2007distributed, ren2007information, kekatos2013distributed, tsianos2012consensus}. The distributed equivalent of problem \req{mainproblem} can be formulated by defining local copies of the decision variable at each agent and rewriting the problem as the following \textit{consensus} problem
\begin{equation} \min_{x_i\in\mathbb{R}} \sum_{i=1}^n f_i(x_i)\quad\mbox{s.t}\quad x_i=x_j \quad\forall{(i,j)\in\mathcal{E}}.\label{consOpt}\end{equation} In this distributed setting, each agent balances the goal of minimizing its local objective function and ensuring that its decision variable is equal to all of its neighbors. } \par

One main category of distributed algorithms for solving problem \req{consOpt} is first order primal methods that can parallelize computations across multiple processors. These methods include different variations of distributed (sub)gradient descent (DGD)
\cite{nedic2001convergence,ram2009incremental,nedic2009distributed,ram2010distributed,nedic2011asynchronous,chen2012fast,jakovetic2014fast}.  In DGD-based methods, each agent updates its local estimate of the solution through a combination of a local gradient decent step and a consensus step (weighted average with neighbors variables). In order to converge to the exact solution, these algorithms need to use a diminishing stepsize, which results in a slow rate of convergence. A faster convergence to a neighborhood of the exact solution can be achieved by using a fixed stepsize. Recently, a class of algorithms based on DGD has been developed in \cite{berahas2018balancing}, which uses increasing number of consensus steps (linear increase with iteration number) to achieve exact linear convergence with fixed stepsize. \par  Another line of distributed optimization is inspired by Lagrange multiplier methods for constrained optimization. Specifically, Method of Multipliers (MM), based on augmented Lagrangian, has nice convergence guarantees in centralized setting  \cite{hestenes1969multiplier, bertsekas2014constrained}. Each iteration of MM includes minimizing the augmented Lagrangian in the primal space, followed by a dual variable update. This method is not implementable in a distributed setting due to the nonseparable augmentation term. Decentralized versions of augmented Lagrangian methods and alternating direction method of multiplier (ADMM) are studied in  \cite{bertsekas2011incremental, boyd2011distributed, wei2012distributed, mota2013d, wei20131, xiao2014proximal, shi2014linear, ling2015dlm, mokhtari2015decentralized, iutzeler2016explicit, jakovetic2015linear,wu2018decentralized}. Although these algorithms involve more computational complexity compared to primal methods, they guarantee convergence to the exact solution. Specifically algorithms in \cite{shi2014linear,ling2015dlm,mokhtari2015decentralized, nedic2018improved} are shown to converge linearly under the assumptions of strong convexity and Lipschitz gradient. %However, each update in the primal space usually requires solving a minimization problem, which can be computationally expensive.  
Some other algorithms like EXTRA \cite{shi2015extra}, DIGing \cite{nedic2017achieving}, and DSA \cite{mokhtari2016dsa} can achieve exact convergence with a fixed stepsize without explicitly using dual variables. These algorithms, however, can be viewed as augmented Lagrangian primal-dual methods with a single gradient step in the primal space. These methods enjoy less computational complexity compared to ADMM-based methods and are proved to converge with a linear rate, under standard assumptions. \par Besides these first order methods, second order distributed algorithms are studied in \cite{mokhtari2015network, mokhtari2016decentralized, tutunov2016distributed, eisen2017decentralized,mansoori2017superlinearly}. These methods use Newton-type updates to achieve faster convergence and thus involve more computational complexity than first order methods. \par 
While existing distributed first order primal-dual algorithms can achieve exact convergence, they either suffer from slower convergence compared to MM or they require minimization at each iteration. In this paper we develop a class of distributed first order primal-dual methods, which does not require exact minimization at each iteration but allows for multiple primal steps per iteration.  This general framework makes it possible to control the trade off between the performance and the execution complexity in primal-dual algorithms.\par 
We develop our algorithm based on a general form of augmented Lagrangian, which is flexible in the augmentation term, and  prove global linear rate of convergence with a fixed stepsize, under the assumptions of strong convexity and Lipschitz gradient. Our framework includes EXTRA and DIGing as special cases for specific choices of augmented Lagrangian function and one primal step per iteration. \par The rest of this paper is organized as follows: Section II describes the development of our general framework. Section III contains the convergence analysis. Section IV presents the numerical experiments and Section V contains the concluding remarks. 
\section{Algorithm Development} 
To develop our algorithm, we express problem \req{consOpt} in the following compact form 
\be\min_x f(x)=\sum_{i=1}^n f_i(x_i)\quad\mbox{s.t}\quad Ax=0, \label{consensus}\ee where $x=[x_1, x_2, ..., x_n]^\prime\in\mathbb{R}^n$, \footnote{\fm{For a vector $x$ we denote its transpose by $x^\prime$ and for a matrix $A$ we denote its transpose by $A^\prime$.}} and $Ax=0$ represents all equality constraints. Matrix $A\in\mathbb{R}^{e\times n}$, $e=\vert \mathcal{E}\vert$, is the edge-node incidence matrix of the network graph and its null space is spanned by a vector of all ones.
%, i.e., $Null(A)=span(\mathbbm{1})$. 
\fm{Row $l$ of matrix $A$ corresponds to edge $l$, connecting vertices $i$ and $j$, and has $+1$ in column $i$ and $-1$ in column $j$ (or vice versa) and $0$ in all other columns \cite{bertsimas1997introduction}.} %and contains one $1$ and one $-1$, i.e., \[A_{li}=\begin{cases} +1, \quad \mbox{i is the lower indexed agent connected to \textit{l}}\\-1\quad\mbox{i is the higher indexed agent connected to \textit{l}}\\0 \quad\mbox{otherwise}\end{cases}\] 
%about matrix A?? %assumption
\begin{remark} Other choices for matrix $A$ include weighted incidence matrix \cite{wu2018decentralized}, graph Laplacian \cite{tutunov2016distributed} matrix, and weighted Laplacian matrix \cite{mokhtari2016decentralized, berahas2018balancing}.
\end{remark}
We denote by $x^*=[\tilde{x}^*, \tilde{x}^*,...,\tilde{x}^*]'$ the minimizer of problem \req{consensus} and we focus on developing a distributed algorithm which converges to $x^*$. 
In order to achieve exact convergence, we develop our algorithm based on the Lagrange multiplier methods. We form the following augmented Lagrangian \be L(x,\lambda)=f(x)+\lambda^\prime Ax+\frac{1}{2}x^\prime Bx, \label{auglag}\ee  
where $\lambda\in\mathbb{R}^e$ is the vector of Lagrange multipliers. Every dual variable $\lambda_l$ is associated with an edge $l=(i,j)$ and thus coupled between two agents. We assume that one of the two agents $i$ or $j$  choose to update the dual variable $\lambda_l$ through negotiation. The set of dual variables that agent $i$ updates is denoted by $\Lambda_i$ . We adopt the following assumptions on our problem.

\begin{assumption} \label{funcprop}
The local objective functions $f_i(x)$ are $m-$ strongly convex, twice differentiable, and $L-$ Lipschitz gradient.
\end{assumption}

\begin{assumption}
Matrix $B\in\mathbb{R}^{n\times n}$ is a symmetric positive semidefinite matrix, has the same null space as matrix $A$, i.e., $Bx=0$ only if $Ax=0$, and is compatible with network topology, i.e., $B_{ij}\neq 0$ only if $(i,j)\in\mathcal{E}$.
\end{assumption}
We assume these conditions hold for the rest of the paper. The first assumption requires the eigenvalues of the Hessian matrix of local objective functions to be bounded, i.e.,  $mI\preceq\nabla^2 f_i(x)\preceq LI$ for all $i$. This is a standard assumption in proving global linear rate of convergence. The second assumption requires matrix $B$ to represent the network topology, which is needed for distributed implementation. The other assumptions on matrix $B$ are required for convergence guarantees. Some examples of matrix $B$ are $B=A^\prime A=L$, with $L$ being the graph Laplacian matrix, and weighted Laplacian matrix. 
\par
One important strand of Lagrange multiplier methods is based on the Method of Multipliers (MM), which requires solving a minimization problem in the primal space, followed by a dual update at each iteration. Despite the nice convergence properties of MM (global linear convergence for strongly convex objective functions), its primal update is computationally costly (because of the exact minimization) and cannot be implemented in a distributed manner (because $x^\prime Bx$ is not separable).  An alternative is to replace the exact minimization in primal space with a single gradient step, which results in the following primal-dual update
\begin{equation}\begin{aligned}&x^{k+1}=x^{k}-\alpha\nabla_xL(x^k,\lambda^k)\\&= x^k-\alpha\nabla f(x^k)-\alpha A^\prime\lambda^k-\alpha B x^k,\label{Primalupdate}\end{aligned}\end{equation}
\[\lambda^{k+1}=\lambda^{k}+\beta\nabla_\lambda L(x^{k+1},\lambda^k)=\lambda^k+\beta Ax^{k+1},\]
\fm{where $\alpha$ and $\beta$ are constant stepsize parameters.}
%We note that since matrices $A$ and $B$ are both neighbor sparse, the above iteration can be implemented in a decentralized manner. \\
The above iterate is implementable in a distributed manner because $\nabla f(x)$ is separable and thus locally available at agents and the terms $\lambda^\prime A$, $Bx$ and $Ax$ can be computed at each agent through communication with neighbors (because matrices $A$ and $B$ represent the graph topology). Although different variations of the above iteration have linear convergence guarantees \cite{shi2015extra,nedic2017achieving}, the linear convergence of MM is shown to be faster \cite{mokhtari2016decentralized}. One interesting question is whether by having multiple updates in the primal space, the performance of the algorithm gets closer to MM. \fm{By looking carefully through Eq. \req{Primalupdate}, we notice that each primal update at each agent requires one gradient evaluation and one round of communication with the neighbors.}
\begin{algorithm}
\caption{The general framework} \label{MCPD}
\begin{algorithmic}
\State Initialization: for $i=1, 2, ..., n$ each agent $i$ picks $x_i^0$, sets $\lambda_{l_i}^0=0\quad\forall{\lambda_{l_i}\in\Lambda_i}$, and determines $\alpha$, $\beta$, and $T<\infty$
\For{$k=1,2,...$}
\State \[x_i^{k+1,0}=x_i^k\]
\For{$t=1,2,...,T$}
\State\begin{align*}&x_i^{k+1,t}=x_i^{k+1,t-1}-\alpha\nabla f_i(x_i^{k})-\alpha\sum_{l=1}^e A^\prime_{il}\lambda_l^k\\&-\alpha\sum_{j=1}^n B_{ij}x_j^{k+1,t-1}\end{align*}
\EndFor
%\State \textbf{end for}
\State \[x_i^{k+1}=x_i^{k+1,T}\]
\State \[\lambda_{l_i}^{k+1}=\lambda_{l_i}^{k}+\beta\sum_{j=1}^n A_{l_ij}x_j^{k+1}\quad\forall{\lambda_{l_i}\in\Lambda_i}.\]
\EndFor
\end{algorithmic}
\end{algorithm}
A follow up question is whether in an algorithm with multiple primal updates per iteration, it is necessary or efficient to recompute the gradient or recommunicate with neighbors at each update.
\par Motivated by these questions, we present an improved class of distributed primal-dual algorithms for solving problem \req{consensus}, which allows for multiple primal updates at each iteration. \sm{To avoid computational complexity, in our algorithm, the gradient is evaluated once per iteration and is used for all primal updates in that iteration.} In our proposed framework in Algorithm \ref{MCPD}, at each iteration $k$, agent $i$ computes its local gradient $\nabla f_i(x_i^k)$, and performs a predetermined number ($T$) of primal updates by \ew{repeatedly communicating with neighbors without recomputing its gradient}. Each agent $i$ then updates its corresponding dual variables $\lambda_{l_i}^{k}$ by using local $x_i^{k+1,T}$ and $x_j^{k+1,T}$ from its neighbors. 
\par \fm{Under Assumption \ref{funcprop}, there exists a unique optimal solution $\tilde{x}^*$ for problem \req{mainproblem} and thus a unique $x^*$ \ew{exists}, at which the function value is bounded. Moreover, since $Null(A)\neq\emptyset$, the Slater's condition is satisfied. Consequently, strong duality holds and a dual optimal solution $\lambda^*$ exists.} We note that the projection of $\lambda^k$ in the null space of matrix $A^\prime $ would not affect the performance of algorithm, and if the algorithm starts at $\lambda=0$, then all the iterates $\lambda^k$ are in the column space of $A$ and hence orthogonal to null space of $A^\prime $. Hence, the optimal dual solution is not uniquely defined, since for any optimal dual solution $\lambda^*$, the dual solution $\lambda^*+u$, where $u$ is in the null space of $A^\prime $ is also optimal. Without loss of generality, we assume that in Algorithm \ref{MCPD}  $\lambda^0=0$, and when we refer to an optimal dual solution $\lambda^*$, we assume its projection onto the null space of $A^\prime $ is $0$. We note that $(x^*,\lambda^*)$ is a fixed point of Algorithm \ref{MCPD}.
% We note that since matrices $A$ and $B$ are neighbor sparse, computing the summation terms in both primal and dual updates only requires local information and communication with neighbors. 
\section{Convergence Analysis}
In order to analyze the convergence properties of our algorithm, we first express the primal-dual update in Algorithm \ref{MCPD} in the following compact form  
\be 
x^{k+1}=(I-\alpha B)^Tx^k-\alpha C\nabla f(x^k)-\alpha CA^\prime \lambda^k, \label{eq:xUpdate}
\ee
\be\lambda^{k+1}=\lambda^k+\beta Ax^{k+1}, \label{eq:muUpdate}\ee
 where $I$ denotes the identity matrix and $C=\sum_{t=0}^{T-1}(I-\alpha B)^t.$  We next proceed to prove the linear convergence rate for our proposed framework. In Lemmas \ref{lemma:Cprop} and \ref{lemma:gradF} we establish some key relations which we use to derive two fundamental inequalities in Lemmas \ref{FundamentalIneq} and \ref{lambdabound}. Finally we use these key inequalities to prove the global linear rate of convergence in Theorem \ref{thm:linconv}. \fm{In the following analysis, we denote by \ew{$\rho(\cdot)$} the largest eigenvalue of a symmetric positive semidefinite matrix and we define matrices $M$ and $N$ as follows \be M=C^{-1}(I-\alpha B)^T\quad\mbox{and}\quad N=\frac{1}{\alpha }(C^{-1}-M). \label{MNdef}\ee} In the next lemma we show that matrix $C$ is invertible and thus matrices $M$ and $N$ are well-defined.\begin{lemma}\label{lemma:Cprop} Consider the symmetric positive semidefinite matrix $B$ and matrices $C$, $M$, and $N$. If we choose $\alpha$ such that $I-\alpha B$ is positive definite, i.e., $\alpha <\frac{1}{\rho(B)}$, then matrix $C$ is invertible and symmetric, matrix $N$ is symmetric positive semidefinite, and matrix $M$ is symmetric positive definite with $\frac{(1-\alpha \rho(B))^T}{\sum_{t=0}^{T-1}(1-\alpha \rho(B))^t}I\preceq M\preceq\frac{1}{T}I.$
\end{lemma} %PD condition for I-\alpha B might be too conservative
\begin{proof} Since $I-\alpha B$ is symmetric, it can be written as $I-\alpha B=VZV^\prime $, where $V\in\mathbb{R}^{n\times n}$ is an orthonormal matrix, i.e., $VV^\prime=I$, whose $i^{th}$ column $v_i$ is the eigenvector of $(I-\alpha B)$ and $v_i^\prime v_t=0$ for $i\neq t$ and $Z$ is the diagonal matrix whose diagonal elements, $Z_{ii}=\mu_i >0$, are the corresponding eigenvalues. We also note that since $V$ is an orthonormal matrix, $(I-\alpha B)^t=VZ^tV^\prime$. We have
\[C=\sum_{t=0}^{T-1}(I-\alpha B)^t=V\big(\sum_{t=0}^{T-1}Z^t\big)V^\prime =VQV^\prime,\]
Hence matrix $C$ is symmetric. We note that matrix $Q$ is a diagonal matrix with $Q_{ii}=1+\sum_{t=1}^{T-1}\mu_i ^t$. Since $\mu_i >0$ for all $i$, $Q_{ii}\neq 0$ and thus $Q$ is invertible and we have %if  $\sum_{t=1}^{T-1}\mu_i ^t\neq -1$ for all $i$, 
$C^{-1}=VQ^{-1}V^\prime.$ We also have $M =C^{-1}(I-\alpha  B)^T=VQ^{-1}V^\prime VZ^TV^\prime =VQ^{-1}Z^TV^\prime =VPV^\prime$
%\begin{align*}M &=C^{-1}(I-\alpha  B)^T=VQ^{-1}V^\prime VZ^TV^\prime \\&=VQ^{-1}Z^TV^\prime =VPV^\prime ,\end{align*}
where $P$ is a diagonal matrix with $P_{ii}=\frac{\mu_i ^T}{1+\sum_{t=1}^{T-1}\mu_i ^t}$, consequently, matrix $M$ is symmetric.
We next find the smallest and largest eigenvalues of matrix $M$. We note that since $P_{ii}$ is increasing in $\mu_i $, the smallest and largest eigenvalues of $M$ can be computed using the smallest and largest eigenvalues of $I-\alpha B$. We have $0\preceq B\preceq \rho(B)I$, where $\rho(B)$ is the largest eigenvalue of matrix $B$. Therefore, the largest and smallest eigenvalues of $I-\alpha B$ are $1$ and $1-\alpha \rho(B)$ respectively. Hence, \[\frac{(1-\alpha \rho(B))^T}{\sum_{t=0}^{T-1}(1-\alpha \rho(B))^t}I\preceq M\preceq\frac{1}{T}I.\] We next use the eigenvalue decomposition of matrices $C^{-1}$ and $M$ to obtain \[C^{-1}-M=VQ^{-1}V^\prime -VPV^\prime =V(Q^{-1}-P)V^\prime,\]
 %\[C^{-1}-M=VQ^{-1}V^\prime -VPV^\prime =V(Q^{-1}-P)V^\prime,\]  
 where $Q^{-1}-P$ is a diagonal matrix, and its $i^{th}$ diagonal element is equal to $\frac{1-\mu_i ^T}{1+\sum_{t=1}^{T-1}\mu_i ^t}$. Since $0<\mu_i \leq 1$ for all $i$, $\frac{1-\mu_i ^T}{1+\sum_{t=1}^{T-1}\mu_i ^t}\geq 0$ and hence $N$ is symmetric positive semidefinite.\end{proof} 
\begin{comment}
\begin{lemma}\label{lemma:NormSquare}
	For any vectors $a$, $b$, and scalar $\xi>1$, we have
	\[(a+b)^\prime (a+b) \leq \frac{\xi}{\xi-1}a^\prime a + \xi b^\prime b.\]
\end{lemma}
\begin{proof}
Since $\xi>1$, we have $\frac{\xi-1}{\xi}+\frac{1}{\xi}=1$ and we can write the right hand side as
\begin{align*}
 &\frac{\xi}{\xi-1}a^\prime a + \xi b^\prime b= (\frac{\xi}{\xi-1}a^\prime a + \xi b^\prime b)(\frac{\xi-1}{\xi}+\frac{1}{\xi})=\\& a^\prime a+b^\prime b+\frac{1}{\xi-1}a^\prime a+(\xi-1)b^\prime b.
\end{align*}

We also have that 
\begin{align*} & 0\leq \left(\sqrt{\frac{1}{\xi-1}}a-\sqrt{\xi-1}b\right)^\prime \left(\sqrt{\frac{1}{\xi-1}}a-\sqrt{\xi-1}b\right) \\& =  \frac{1}{\xi-1}a^\prime a+(\xi-1)b^\prime b - 2a^\prime b,\end{align*} which implies that 
$\frac{1}{\xi-1}a^\prime a+(\xi-1)b^\prime b\geq 2a^\prime b .$

We can then combine this into the previous equality and have 
\begin{align*}
\frac{\xi}{\xi-1}a^\prime a + \xi b^\prime b\geq a^\prime a+b^\prime b+2a^\prime b = (a+b)^\prime (a+b),
\end{align*}
which establishes the desired relation.
\end{proof}
\end{comment}
\begin{lemma}\label{lemma:gradF} Consider the primal-dual iterates as in Algorithm \ref{MCPD} and recall the definitions of matrices $M$ and $N$ from Eq. \req{MNdef}, if $\alpha <\frac{1}{\rho(B)}$, then 
	\begin{align*} &\alpha (\nabla f(x^k)- \nabla f(x^*) )=  M(x^k-x^{k+1}) +\\& \alpha  (\beta  A^\prime A-N) (x^{k+1} - x^*) -\alpha  A^\prime (\lambda^{k+1} -\lambda ^*),\end{align*}
\end{lemma}
\begin{proof}
	At each iteration, from Eq. \req{eq:xUpdate} we have
	\[ \alpha C\nabla f(x^k)  = (I-\alpha B)^Tx^k-x^{k+1} -\alpha C A^\prime  \lambda^k.\]
	Moreover, from Eq. \req{eq:muUpdate} we have
	$\lambda ^k = \lambda ^{k+1} - \beta  A x^{k+1}.$ We can substitute this expression for $\lambda^k$ into the previous equation and have
	\begin{equation}\begin{aligned} &\alpha C\nabla f(x^k)  =  (I-\alpha B)^Tx^k-x^{k+1}  \\ &-\alpha C A^\prime  (\lambda ^{k+1} - \beta  A x^{k+1}) =(I-\alpha  B)^T(x^k-x^{k+1}) +\\& \big(\alpha  \beta  CA^\prime A-I+(I-\alpha  B)^T\big) x^{k+1} -\alpha C A^\prime \lambda^{k+1},\label{gradeq}\end{aligned}\end{equation} where we added and subtracted a term of $(I-\alpha  B)^T x^{k+1}$.
Since an optimal solution pair $(x^*, \lambda^*)$ is a fixed point of the algorithm update, we also have
	\[ \alpha  C\nabla f(x^*)  =   \big(\alpha \beta  CA^\prime A-I+(I-\alpha  B)^T\big) x^{*} -\alpha  C A^\prime \lambda^{^*}.\]
	We then subtract the above inequality from Eq. \req{gradeq} and multiply both sides by $C^{-1}$ [c.f. Lemma \ref{lemma:Cprop}],  to obtain
	\begin{align*} &\alpha  (\nabla f(x^k)- \nabla f(x^*) )=  C^{-1}(I-\alpha  B)^T(x^k-x^{k+1}) \\&+  \alpha \Big(\beta  A^\prime A-\frac{1}{\alpha }\big(C^{-1}-C^{1}(I-\alpha  B)^T\big)\Big) (x^{k+1} - x^*) \\&-\alpha  A^\prime (\lambda^{k+1} -\lambda ^*).\end{align*}\end{proof}
\begin{lemma} \label{FundamentalIneq}
Consider the primal-dual iterates as in Algorithm \ref{MCPD} and recall the definition of matrices $M$ and $N$ from Eq. \req{MNdef}. If $\alpha<\frac{1}{\rho(B)}$, we have 
\ew{for any $\eta>0$},
\begin{equation}\begin{aligned}  &(x^{k+1}-x^*)^\prime (2\alpha  m I-\alpha \eta I +2\alpha (N-\beta  A^\prime A))   (x^{k+1}-x^*)+\\&\frac{\alpha }{\beta }\norm{\lambda^{k+1}-\lambda^k}^2+ (x^{k+1}-x^k)^\prime (M- \frac{\alpha  L^2}{\eta}I) (x^{k+1}-x^k)\\&\leq
 (x^k-x^*)^\prime M (x^k -x^*) - (x^{k+1}-x^*)^\prime M(x^{k+1}-x^*)
 \\&+ \frac{\alpha }{\beta }\left(\norm{\lambda^k-\lambda^*}^2 - \norm{\lambda^{k+1}-\lambda^*}^2\right).\label{fund_ineq}
 \end{aligned}\end{equation}
\end{lemma}
\begin{proof}
From strong convexity of function $f(x)$, we have
\begin{align*} & 2\alpha m\norm{x^{k+1}-x^*}^2 \leq\\& 2\alpha  (x^{k+1}-x^*)^\prime (\nabla f(x^{k+1}) - \nabla f(x^*)) \\&= 2\alpha  (x^{k+1}-x^*)^\prime (\nabla f(x^{k+1})-\nabla f(x^k)) \\&+ 2\alpha  (x^{k+1}-x^*)^\prime (\nabla f(x^k)-\nabla f(x^*)),\end{align*} where we add and subtract a term $(x^{k+1}-x^*)^\prime \nabla f(x^k)$.
We can substitute the equivalent expression of $ \alpha (\nabla f(x^k)- \nabla f(x^*) )$ from  Lemma \ref{lemma:gradF} and have
\begin{equation}\begin{aligned} & 2\alpha  m \norm{x^{k+1}-x^*}^2\leq2\alpha  (x^{k+1}-x^*)^\prime\times\\&  (\nabla f(x^{k+1})-\nabla f(x^k)) + 2(x^{k+1}-x^*)^\prime\big(M(x^k-x^{k+1}) \\& + \alpha  (\beta  A^\prime A-N) (x^{k+1} - x^*) -\alpha  A^\prime (\lambda^{k+1} -\lambda ^*)\big) 
=\\& 2\alpha  (x^{k+1}-x^*)^\prime (\nabla f(x^{k+1})-\nabla f(x^k))+2\alpha  (x^{k+1}-x^*)^\prime\\ &  (\beta  A^\prime A - N)(x^{k+1}-x^*)+2(x^{k+1}-x^*)^\prime M(x^k-x^{k+1})\\&  -2\alpha (x^{k+1}-x^*)^\prime  A^\prime (\lambda^{k+1} -\lambda ^*). \label{eq:update1} \end{aligned}\end{equation}
\sm{From Young's inequality we have}  $2\alpha  (x^{k+1}-x^*)^\prime (\nabla f(x^{k+1})-\nabla f(x^k)) \leq \alpha \eta \norm{x^{k+1}-x^*}^2+\frac{\alpha }{\eta}\norm{\nabla f(x^{k+1})-\nabla f(x^k)}^2$ for all $\eta>0$. 
		%\begin{align*}& 2\alpha  (x^{k+1}-x^*)^\prime (\nabla f(x^{k+1})-\nabla f(x^k)) \leq\\& \alpha \eta \norm{x^{k+1}-x^*}^2+\frac{\alpha }{\eta}\norm{\nabla f(x^{k+1})-\nabla f(x^k)}^2.\end{align*}
This can be further reduced to $2\alpha  (x^{k+1}-x^*)^\prime (\nabla f(x^{k+1})-\nabla f(x^k)) \leq \alpha \eta \norm{x^{k+1}-x^*}^2+\frac{\alpha  L^2 }{\eta}\norm{x^{k+1}-x^k}^2$, by Lipschitz gradient property of function $f$.
%\begin{align*}& 2\alpha  (x^{k+1}-x^*)^\prime (\nabla f(x^{k+1})-\nabla f(x^k)) \leq \\& \alpha \eta \norm{x^{k+1}-x^*}^2+\frac{\alpha  L^2 }{\eta}\norm{x^{k+1}-x^k}^2,\end{align*}
By dual update Eq. \eqref{eq:muUpdate} and feasibility of $x^*$, we have 
\[Ax^{k+1} = \frac{1}{\beta }(\lambda^{k+1}-\lambda^k),\quad Ax^* = 0.\]
These two equations combined yields
\[\alpha (x^{k+1}-x^*)^\prime  A^\prime (\lambda^{k+1} -\lambda ^*) = \frac{\alpha }{\beta }(\lambda^{k+1}-\lambda^k)^\prime (\lambda^{k+1}-\lambda^*).\]
Hence we can rewrite the right hand side of Eq. \eqref{eq:update1} as  
\begin{align*} & 2\alpha (x^{k+1}-x^*)^\prime (\beta A^\prime A-N)(x^{k+1}-x^*)+ \alpha \eta\\& \Vert x^{k+1}- x^*\Vert^2+\frac{\alpha  L^2}{\eta}\norm{x^{k+1}-x^k}^2 - 2(x^{k+1}-x^*)^\prime\\&\times M(x^{k+1}-x^k) -2\frac{\alpha }{\beta }(\lambda^{k+1}-\lambda^k)^\prime (\lambda^{k+1}-\lambda^*).
 \end{align*}
 We now focus on the last two terms. First, since matrix $M$ is symmetric, we have
 \begin{align*}& -2(x^{k+1}-x^*)^\prime M(x^{k+1}-x^k)  = (x^k-x^*)^\prime M (x^k -x^*) -\\& (x^{k+1}-x^*)^\prime M(x^{k+1}-x^*) - (x^{k+1}-x^k)^\prime M (x^{k+1}-x^k).\end{align*}
 Similarly, we have $-2\frac{\alpha }{\beta }(\lambda^{k+1}-\lambda^k)^\prime (\lambda^{k+1}-\lambda^*) =$
 \begin{align*} \frac{\alpha }{\beta }\left(\norm{\lambda^k-\lambda^*}^2 - \norm{\lambda^{k+1}-\lambda^*}^2-\norm{\lambda^{k+1}-\lambda^k}^2\right).
  \end{align*}
 Now we combine the terms in the preceding three relations and have 
\begin{align*}&2\alpha  m \norm{x^{k+1}-x^*}^2\leq 2\alpha (x^{k+1}-x^*)^\prime (\beta A^\prime A-N)\times\\&(x^{k+1}-x^*)+ \alpha \eta \norm{x^{k+1}-x^*}^2+\frac{\alpha  L^2}{\eta}\norm{x^{k+1}-x^k}^2+\\& (x^k-x^*)^\prime M(x^k -x^*) - (x^{k+1}-x^*)^\prime M(x^{k+1}-x^*)\\&- (x^{k+1}-x^k)^\prime M (x^{k+1}-x^k)\\&+ \frac{\alpha }{\beta }\left(\norm{\lambda^k-\lambda^*}^2 - \norm{\lambda^{k+1}-\lambda^*}^2-\norm{\lambda^{k+1}-\lambda^k}^2\right).
 \end{align*} By rearranging the terms in the above inequality, we complete the proof.
\end{proof}
\begin{lemma}\label{lambdabound}
Consider the primal-dual iterates as in Algorithm \ref{MCPD} and recall the definition of symmetric matrices $M$ and $N$ from Eq.\req{MNdef} then if $\alpha<\frac{1}{\rho(B)}$, for $d, g, e>1$ we have
\begin{equation}\begin{aligned}\label{ineq:toCompare}
  &\norm{\lambda^{k+1}-\lambda^*}^2\leq \frac{d}{\alpha ^2 s(AA^\prime )} \left(\frac{e}{e-1} \rho(M)^2+e \alpha ^2L^2\right)\\&\times\left((x^k-x^{k+1})^\prime (x^k-x^{k+1}) \right)+ \frac{d}{(d-1)\alpha ^2 s(AA^\prime )}\\&\times \left(\frac{g}{g-1} \alpha ^2 \rho \big((\beta  A^\prime A-N)^2\big)+\alpha ^2gL^2\right)\\&\times\left((x^{k+1} - x^*)^\prime  (x^{k+1} - x^*)\right),
  \end{aligned}\end{equation}
with $s(AA^\prime )$ being the smallest nonzero eigenvalue of the positive semidefinite matrix $AA^\prime$.
\end{lemma}
\begin{proof}
We recall the following relation from Lemma \ref{lemma:gradF}
	\begin{align*}&\alpha  A^\prime (\lambda^{k+1} -\lambda ^*)=  M(x^k-x^{k+1}) + \\& \alpha  (\beta  A^\prime A-N) (x^{k+1} - x^*) - \alpha (\nabla f(x^k)- \nabla f(x^*) ).\end{align*}	
Thus we have
\begin{align*}& \norm{\alpha  A^\prime (\lambda^{k+1} -\lambda ^*)}^2 = \big\Vert M(x^k-x^{k+1}) + \\& \alpha  (\beta  A^\prime A-N) (x^{k+1} - x^*) - \alpha (\nabla f(x^k)- \nabla f(x^*) )\big\Vert^2.\end{align*}
We can add and subtract a term of $\nabla f(x^{k+1})$ and equivalently express the above relation as
\begin{align*}&\norm{\alpha  A^\prime (\lambda^{k+1} -\lambda ^*)}^2 =\\& \big\Vert M(x^k-x^{k+1}) +  \alpha  (\beta  A^\prime A-N) (x^{k+1} - x^*) -\\&\alpha \left(\nabla f(x^{k})- \nabla f(x^{k+1} )\right)- \alpha \left(\nabla f(x^{k+1})- \nabla f(x^*) \right)\big\Vert^2.\end{align*}
By using the fact that $(a+b)^\prime (a+b) \leq \frac{\xi}{\xi-1}a^\prime a + \xi b^\prime b$, for any vectors $a$, $b$, and scalar $\xi>1$, we have
for any scalars $d,e,g>1$, 
\begin{align*}&\norm{\alpha  A^\prime (\lambda^{k+1} -\lambda ^*)}^2 \leq\\& d\Big(\frac{e}{e-1} (x^k-x^{k+1})^\prime M^2(x^k-x^{k+1}) +e \alpha ^2\times\\& \left(\nabla f(x^{k})- \nabla f(x^{k+1} )\right)^\prime \left(\nabla f(x^{k})- \nabla f(x^{k+1} )\right)\Big) +\\& \frac{d}{d-1}\Big(\frac{g}{g-1} \alpha ^2(x^{k+1} - x^*)^\prime   (\beta  A^\prime A-N)^2 (x^{k+1} - x^*) \\&+ \alpha ^2g\left(\nabla f(x^{k+1})- \nabla f(x^*) \right)^\prime \left(\nabla f(x^{k+1})- \nabla f(x^*) \right)\Big).
\end{align*}

Since $\lambda^0=0$ and $\lambda^{k+1} = \lambda^k + \beta  Ax^{k+1}$, we have that $\lambda^k$ is in the column space of $A$ and hence orthogonal to the null space of $A^\prime $, therefore, we have
$\norm{\alpha  A^\prime (\lambda^{k+1} -\lambda ^*)}^2\geq \alpha ^2s(AA^\prime )\norm{\lambda^{k+1}-\lambda^*}^2.$ Using the above two relations and Lipschitz gradient property of function $f$ , we have
 \begin{align*}
 &\alpha ^2 s(AA^\prime )\norm{\lambda^{k+1}-\lambda^*}^2\leq d\Big((x^k-x^{k+1})^\prime \\&\times \left[\frac{e}{e-1} M^2+e \alpha ^2L^2I\right](x^k-x^{k+1}) \Big) + \frac{d}{d-1}\times\\&\Big((x^{k+1} - x^*)^\prime  \left[\frac{g}{g-1} \alpha ^2 (\beta  A^\prime A-N)^2+\alpha ^2gL^2I\right] \\&\times(x^{k+1} - x^*)\Big).
 \end{align*}
 By using the facts that $M^2\preceq\rho(M^2) I$ and $(\beta  A^\prime A-N)^2\preceq\rho\big((\beta  A^\prime A-N)^2\big)I$, we complete the proof.
  \end{proof}
\begin{theorem} \label{thm:linconv}
Consider the primal-dual iterates as in Algorithm \ref{MCPD}, recall the definition of matrix $M$ from Eq. \req{MNdef}, and define $ z^k=\begin{bmatrix} x^k \\ \lambda^k \end{bmatrix}, \quad G=\begin{bmatrix} M & \textbf{0} \\ \textbf{0} & \frac{\alpha}{\beta} I \end{bmatrix}.$  If Assumption \ref{funcprop} holds true, and the primal and dual stepsizes satisfy 
\be0<\beta <\frac{2m}{\rho(A^\prime A)}. \label{eq:betabound}\ee \be 0<\alpha <\frac{1-\left(\frac{L^2}{L^2+\eta\rho(B)}\right)^{1/T}}{\rho(B)}, \label{eq:alphabound}\ee with $0<\eta<2m-\beta \rho(A^\prime A)$, then there exists a $\delta>0$ such that 
\[\norm{z^{k+1}-z^*}^2_G\leq\frac{1}{1+\delta}\norm{z^k-z^*}^2_G,\] that is $\norm{z^k-z^*}_G$ converges Q-linearly to $0$  and consequently $\norm{x^k-x^*}_M$ converges R-linearly to $0$.
\end{theorem}
\begin{proof}
To show linear convergence, we will show that
\begin{align*}   \delta&\left((x^{k+1}-x^*)^\prime M (x^{k+1} -x^*) + \frac{\alpha }{\beta }\norm{\lambda^{k+1}-\lambda^*}^2\right)\leq\\&  (x^k-x^*)^\prime M(x^k -x^*) - (x^{k+1}-x^*)^\prime M(x^{k+1}-x^*)\\&
+ \frac{\alpha }{\beta }\left(\norm{\lambda^k-\lambda^*}^2 - \norm{\lambda^{k+1}-\lambda^*}^2\right),
\end{align*}
for some $\delta>0$. By using the result of Lemma \ref{FundamentalIneq} [c.f. Eq. \req{fund_ineq}], it suffices to show that there exists a $\delta>0$ such that
\begin{align*}   &\delta\left((x^{k+1}-x^*)^\prime M (x^{k+1} -x^*) + \frac{\alpha }{\beta }\norm{\lambda^{k+1}-\lambda^*}^2\right) \\ &\leq (x^{k+1}-x^*)^\prime (2\alpha  m I-\alpha \eta I +2\alpha (N-\beta  A^\prime A))\times  \\ & (x^{k+1}-x^*)+ (x^{k+1}-x^k)^\prime (M- \frac{\alpha  L^2}{\eta}I) (x^{k+1}-x^k)\\& +\frac{\alpha }{\beta }\norm{\lambda^{k+1}-\lambda^k}^2.\end{align*}We now use Eq. \req{eq:muUpdate} together with the fact that $Ax^*=0$ to obtain $\norm{\lambda^{k+1}-\lambda^k}^2 =\beta ^2(x^{k+1}-x^*)^\prime (A^\prime A)(x^{k+1}-x^*),
$ which we can substitute into the previous inequality and have that we need to show
\begin{align*}   &\delta\left((x^{k+1}-x^*)^\prime M (x^{k+1} -x^*) + \frac{\alpha }{\beta }\norm{\lambda^{k+1}-\lambda^*}^2\right)\\&\leq(x^{k+1}-x^*)^\prime (2\alpha  m I-\alpha \eta I +2\alpha (N-\beta  A^\prime A))  \\ &\times(x^{k+1}-x^*)+\alpha \beta (x^{k+1}-x^*)^\prime (A^\prime A)(x^{k+1}-x^*)\\&+ (x^{k+1}-x^k)^\prime (M- \frac{\alpha  L^2}{\eta}I) (x^{k+1}-x^k).\end{align*}
We collect the terms and we will focus on showing 
\begin{equation}
\begin{aligned}\label{ineq:linearRateToShow}   &\norm{\lambda^{k+1}-\lambda^*}^2\leq  \frac{\beta }{\delta\alpha }(x^{k+1}-x^k)^\prime (M- \frac{\alpha  L^2}{\eta}I) \\&\times(x^{k+1}-x^k)+\frac{\beta }{\delta\alpha }(x^{k+1}-x^*)^\prime \big(2\alpha  m I-\alpha \eta I +\\&2\alpha (N-\beta  A^\prime A)-\delta M+\alpha \beta  A^\prime A\big)(x^{k+1}-x^*) .\end{aligned}\end{equation}
We compare Eq.\ \eqref{ineq:linearRateToShow} with Eq.\ \eqref{ineq:toCompare} [c.f. Lemma \ref{lambdabound}], and we need to have for some $\delta>0$ 
  \begin{align*}     &\frac{d}{\alpha ^2 s(AA^\prime )} \left(\frac{e\rho (M)^2}{e-1} +e \alpha ^2L^2\right)\big((x^k-x^{k+1})^\prime (x^k-\\&x^{k+1}) \big) + \frac{d}{(d-1)\alpha ^2 s(AA^\prime )} \Big(\frac{g\alpha ^2}{g-1}  \rho \big((\beta  A^\prime A-N)^2\big)+\\&\alpha ^2gL^2\Big)\left((x^{k+1} - x^*)^\prime  (x^{k+1} - x^*)\right)\leq  \frac{\beta }{\delta\alpha }(x^{k+1}-x^k)^\prime\\& \times(M- \frac{\alpha  L^2}{\eta}I) (x^{k+1}-x^k)+\frac{\beta }{\delta\alpha }(x^{k+1}-x^*)^\prime\big(2\alpha  m I\\&-\alpha \eta I +2\alpha (N-\beta  A^\prime A)-\delta M+\alpha \beta  A^\prime A\big)(x^{k+1}-x^*) .\end{align*}
  
  This is satisfied if 
  \begin{align*}&\frac{\beta }{\delta\alpha }(M- \frac{\alpha  L^2}{\eta}I) \succcurlyeq \frac{d}{\alpha ^2 s(AA^\prime )} \left(\frac{e\rho(M)^2}{e-1} +e \alpha ^2L^2\right)I,\\
  &\frac{\beta }{\delta\alpha }(2\alpha  m-\alpha \eta I +2\alpha  N-\alpha \beta  A^\prime A-\delta M)\succcurlyeq \\& \frac{d}{(d-1)\alpha ^2 s(AA^\prime )} \left(\frac{g\alpha ^2 \rho}{g-1}  \big((\beta  A^\prime A-N)^2\big)+g\alpha ^2L^2\right)I.
  \end{align*}
  Since the previous two inequalities holds for all $e, d, g>1$, we can find the parameters $e$ and $g$ to make the right hand sides the smallest, which would give us the most freedom to choose algorithm parameters.  The term $\frac{e}{e-1} \rho(M)^2 + e \alpha ^2L^2$ is convex in $e$ and to minimize it we set derivative to $0$ and have 
$e = 1+ \frac{\rho(M)}{\alpha  L}.$
Similarly, we choose $g$ to be
$ g =1 + \frac{ \sqrt{\rho \big((\beta  A^\prime A-N)^2\big)}}{ L}.$
With these parameter choices, we have
$\left(\frac{e}{e-1} \rho(M)^2+e \alpha ^2L^2\right) = (\rho(M)+ \alpha  L)^2,$ and
$\left(\frac{g}{g-1} \alpha ^2 \rho \big((\beta  A^\prime A-N)^2\big)+g\alpha ^2L^2\right) = \alpha ^2 \Big(\sqrt{\rho\big((\beta  A^\prime A-N)^2\big)}+L\Big)^2.$ The desired relations can now be expressed as $\frac{\beta }{\delta\alpha }(M- \frac{\alpha  L^2}{\eta}I) \succcurlyeq \frac{d}{\alpha ^2 s(AA^\prime )} \big(\rho(M)+ \alpha  L\big)^2I$ and
  $\frac{\beta }{\delta\alpha }(2\alpha  m-\alpha \eta I +2\alpha  N-\alpha \beta  A^\prime A-\delta M)\succcurlyeq  \frac{d}{(d-1) s(AA^\prime )}\Big(\sqrt{\rho\big((\beta  A^\prime A-N)^2\big)}+L\Big)^2I.
  $
By using the fact that $N$ and $A^\prime A$ are  positive semidefinite matrices, and the result of Lemma \ref{lemma:Cprop} to bound eigenvalues of matrix $M$, the desired relations can be satisfied if

  \begin{align*}&\frac{\beta }{\delta\alpha }\Big(\frac{(1-\alpha \rho(B))^T}{\sum_{t=0}^{T-1}(1-\alpha \rho(B))^t}-\frac{\alpha  L^2}{\eta}\Big) \geq\\& \frac{d}{\alpha ^2 s(AA^\prime )} (\frac{1}{T}+ \alpha  L)^2,\\
  &\frac{\beta }{\delta\alpha }(2\alpha  m-\alpha \eta-\alpha \beta \rho(A^\prime A) -\frac{\delta}{T}) \geq \\& \frac{d}{(d-1) s(AA^\prime )} \Big(\sqrt{\rho\big((\beta  A^\prime A-N)^2\big)}+L\Big)^2.
  \end{align*}
For the first inequality, we can multiply both sides by $\delta\alpha $ and rearrange the terms to have
%\[\beta \Big(\frac{(1-\alpha \rho(B))^T}{\sum_{t=0}^{T-1}(1-\alpha \rho(B))^t}-\frac{\alpha  L^2}{\eta}\Big) \geq \frac{d \delta}{ \alpha  s(AA^\prime )} (\frac{1}{T}+ \alpha  L)^2,\]
%i.e., 
\be
\delta \leq \frac{\alpha \beta  \Big(\frac{(1-\alpha \rho(B))^T}{\sum_{t=0}^{T-1}(1-\alpha \rho(B))^t}-\frac{\alpha  L^2}{\eta}\Big)s(AA^\prime )  }{d(\frac{1}{T}+ \alpha  L)^2}.\label{firstbound}
\ee
We can similarly solve for the second inequality, 
\be \delta\leq\frac{ \beta (2\alpha  m-\alpha \eta-\alpha \beta \rho (A^\prime A) )}{ \frac{d\alpha }{(d-1) s(AA^\prime )}   \Big(\sqrt{\rho\big((\beta  A^\prime A-N)^2\big)}+L\Big)^2+\frac{\beta }{T}}.\label{secondbound}
\ee We next show that the upper bounds on $\delta$ given in Eq. \req{firstbound} and Eq. \req{secondbound} are both positive. Since the dual step size $\beta$ satisfies Eq. \req{eq:betabound} and $0<\eta<2m-\beta \rho(A^\prime A)$, the right hand side of Eq. \req{secondbound} is positive. 

Moreover, in order for the right hand side of Eq. \req{firstbound} to be positive we need
\be  \frac{(1-\alpha \rho(B))^T}{\sum_{t=0}^{T-1}(1-\alpha \rho(B))^t}-\frac{\alpha  L^2}{\eta}>0. \label{eq:alpha}\ee
Since $1-\alpha \rho(B)\neq 1$ [c.f. Eq. \req{eq:alphabound}], we have \[\sum_{t=0}^{T-1}(1-\alpha \rho(B))^t=\frac{1-(1-\alpha \rho(B))^T}{1-(1-\alpha \rho(B))}=\frac{1-(1-\alpha \rho(B))^T}{\alpha \rho(B)}.\] Therefore inequality \eqref{eq:alpha} can be written as
\[\frac{(1-\alpha \rho(B))^T}{1-(1-\alpha \rho(B))^T}\alpha \rho(B)-\alpha \frac{L^2}{\eta}>0,\] which holds true for $\alpha$ satisfying Eq. \req{eq:alphabound}.
%\[ 0<\alpha <\frac{\big(1+\frac{\eta\rho(B)}{L^2}\big)^{1/T}-1}{\rho(B)\big(1+\frac{\eta\rho(B)}{L^2}\big)^{1/T}}.\]

Hence, the parameter set is nonempty and thus we can find \begin{align*} & 0<\delta\leq\min\Bigg\lbrace\frac{\alpha \beta  \Big(\frac{(1-\alpha \rho(B))^T}{\sum_{t=0}^{T-1}(1-\alpha \rho(B))^t}-\frac{\alpha  L^2}{\eta}\Big)s(AA^\prime )  }{d(\frac{1}{T}+ \alpha  L)^2},\\& \frac{ \beta (2\alpha  m-\alpha \eta-\alpha \beta \rho (A^\prime A) )}{ \frac{d }{(d-1) s(AA^\prime )}\alpha   (\rho(\beta  A^\prime A-N)+L)^2+\frac{\beta }{T}}\Bigg\rbrace \end{align*} which establishes linear rate of convergence.
\end{proof}	
\ss{\begin{remark}\label{Bchoice} If we choose $B=\beta A^\prime A$, we have $N-\beta A^\prime A=0$, and from the analysis of the above theorem we can see that the upper bound on $\beta$ can be removed. \sm{In this case, by choosing $T=1$, the stepsize upper bound in Eq. \req{eq:alphabound} and the linear rate parameter recover those of EXTRA\cite{shi2015extra}.} %Existing literature suggests that, increasing $\beta$ results in faster convergence in the method of multipliers (exact minimization in primal space).
\end{remark}}
\sm{\begin{remark}\label{Tchoice}
To find an optimal value for the number of primal updates per iteration, $T$, leading to the best convergence rate, we can optimize over various parameters in the analysis. While a general result is quite messy, we can show that the upper bound on $T\alpha$ [c.f. Eq. \req{eq:alphabound}] is increasing in $T$ and approaches $-\ln{\frac{L^2}{L^2+\eta\rho(B)}}$ for large values of $T$. This suggests that the benefits in improving convergence speed from increasing $T$ diminishes for large $T$. 
\end{remark}}
\sm{\begin{remark} In our algorithm, the stepsize parameters are  common among all agents and computing them requires global variables across the network. This global variables can be obtained by applying a consensus algorithm prior to the main algorithm \cite{shi2015extra,wu2018decentralized, mokhtari2015network}.
\end{remark}}
\section{Simulation Results} 
In this section we present some numerical experiments, where we compare the performance of our proposed algorithms with other first order exact methods. We also study the performance of our algorithms on networks with different sizes and topologies. In all experiments we set $B=\beta A^\prime A$ for our algorithm. \par We first consider solving a classification problem using regularized logistic regression. We consider a problem with $K$ training samples that are uniformly distributed over $n=10$ agents in a network with $4-$regular graph, in which agents first form a ring and then each agent gets connected to two other neighbors (one from each side). Each agent $i$ has access to  a batch of data with size $k_i=\floor{\frac{K}{n}}$. This problem can be formulated as follows
\[\min_x f(x)=\frac{\nu}{2}\norm{x}^2+\frac{1}{K}\sum_{i=1}^n\sum_{j=1}^{k_i}\log\big[1+\exp(-v_{ij}u_{ij}^\prime x)\big],\]
where $u_{ij}$ and $v_{ij}$, $j\in\{ 1, 2, ..., k_i\}$ are the feature vector and the label for the data point $j$ associated with agent $i$ and the regularizer $\frac{\nu}{2}\big\Vert x\big\Vert^2$ is added to avoid overfitting. We can write this objective function in the form of $f(x)=\sum_{i=1}^n f_i(x)$, where $f_i(x)$ is defined as 
\[f_i(x)=\frac{\nu}{2n}\norm{x}^2+\frac{1}{K}\sum_{j=1}^{k_i}\log\big[1+\exp(-v_{ij}u_{ij}^\prime x)\big].\]
\begin{figure} 
\vskip 0.2in
  \centering
  \subfloat{\includegraphics[width=0.229\textwidth]{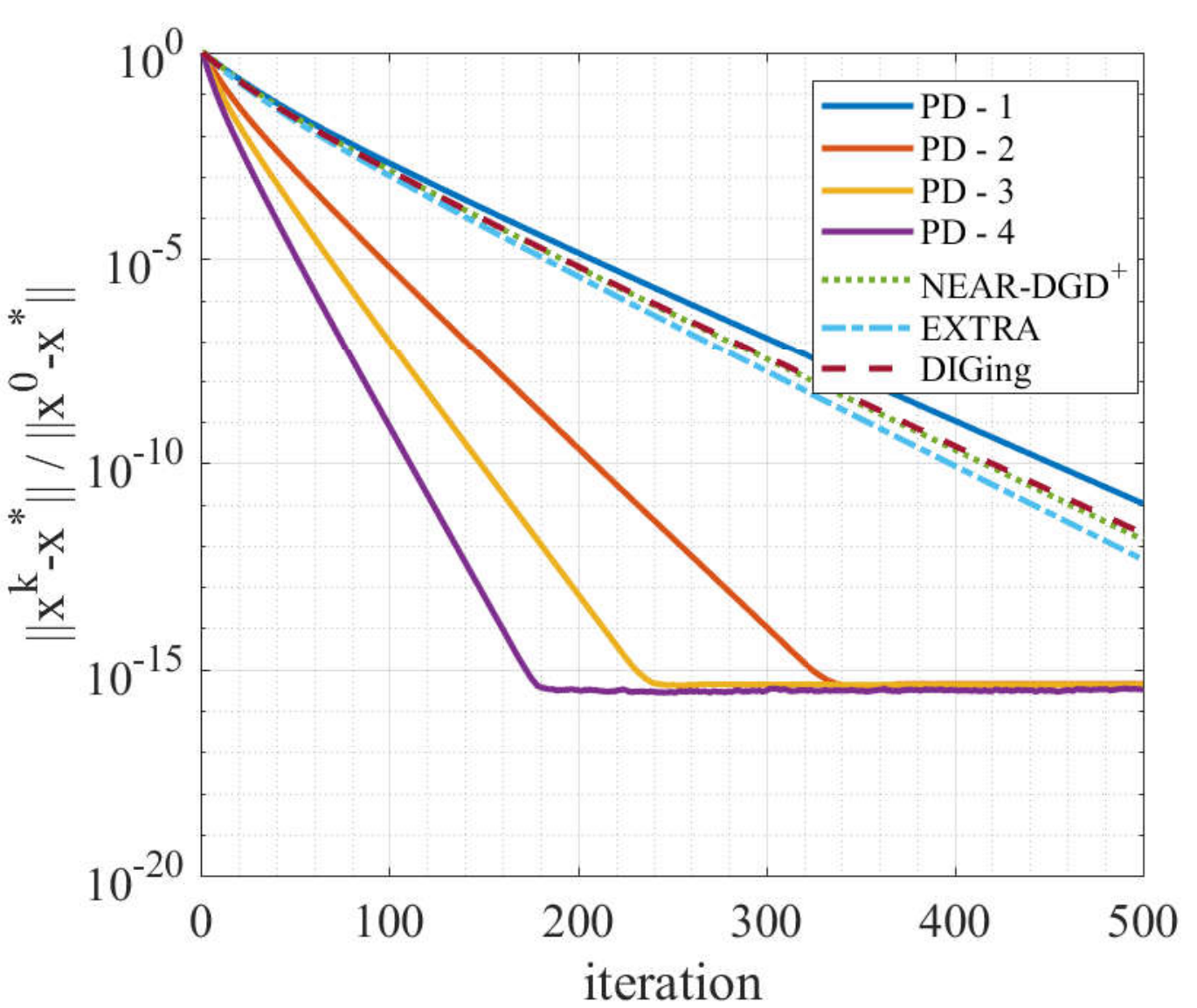}}
  \hfill
  \subfloat{\includegraphics[width=0.23\textwidth]{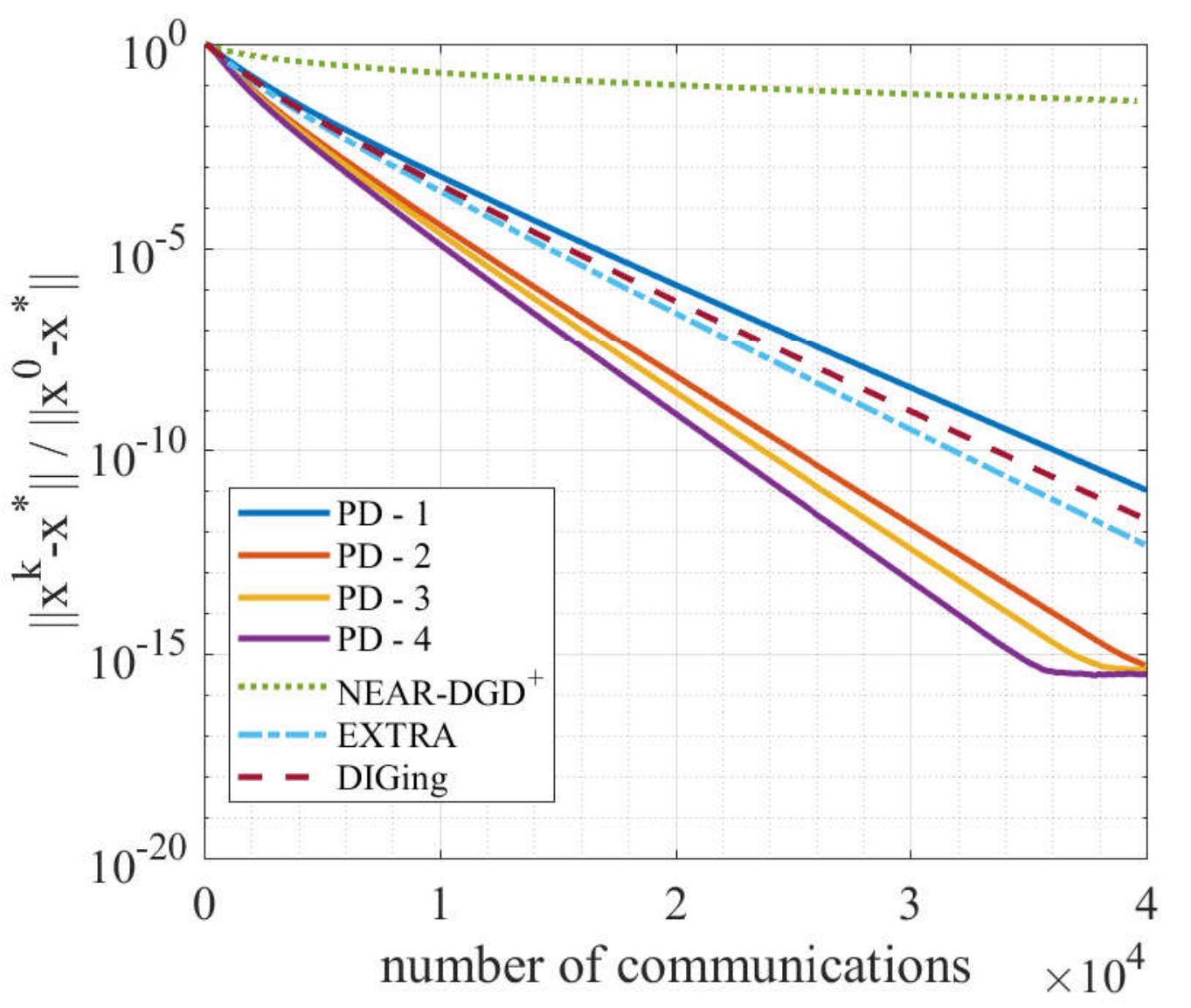}}
  \caption{Performance of primal-dual algorithm with 1-4 primal updates per iteration, NEAR-DGD$^+$, DIGing, and EXTRA in terms of the relative error.} \label{fig1}
  \vskip -0.2in
\end{figure}
\par In order to study the performance of our proposed framework on different networks, we consider $5-30$ agents which are connected with random $4-$regular graphs \fm{(agents first form a ring and then each agent gets connected to two other random agents)}. The objective function at each agent $i$ is of the form $f_i(x)=c_i(x_i-b_i)^2$ with $c_i$ and $b_i$ being integers that are randomly chosen from $[1,10^4]$ and $[1,100]$. We run the simulation for $1000$ random seeds and we plot the average number of steps and communications until the relative error is less than $\epsilon=0.01$, i.e., $\frac{\norm{x^k-x^*}}{\norm{x^0-x^*}}<0.01$ in Figure \ref{fig2}. The method of multipliers (centralized implementation) is also included as a benchmark. The primal stepsize parameter $\alpha$ at each seed is chosen based on the theoretical bound given in Eq. \req{eq:alphabound} and the dual stepsize is $\beta=T$. \fm{We observe that regardless the size of network,  increasing the number of primal updates per iteration improves the performance of the algorithm.} \ew{We also observe that as the network size grows, the number of steps to optimality of our proposed method grows sublinearly and the communication needed increases almost linearly.} 
In our simulations, we use the \textit{mushrooms} dataset \cite{chang2011libsvm},  with $8124$ data points, distributed uniformly over $10$ agents. Each data point has a feature vector of size $112$ and a label which is either $1$ or $-1$. \fm{In Figure \ref{fig1} we compare the performance of our primal-dual algorithms [c.f. Algorithm \ref{MCPD}] with $T=1, ..., 4$ (represented by PD - 1 to PD - 4), with EXTRA \cite{shi2015extra}, DIGing \cite{nedic2017achieving}, and NEAR-DGD$^+$ \cite{berahas2018balancing}, in terms of relative error, $\frac{\norm{x^k-x^*}}{\norm{x^0-x^*}}$, with respect to number of iterations and number of communications.} The benchmark $x^*$ is computed using \textit{minFunc} software \cite{schmidt2005minfunc} and the stepsize parameters are manually tuned for each algorithm. We can clearly see that increasing the number of primal updates improves the performance of the algorithm, while incurring a  higher communication cost. \sm{In our experiments, we observed that by increasing $T$ to larger numbers, the performance of the algorithms does not improve much, which can be explained through Remark \ref{Tchoice} and also the effect of the outdated gradients.} \fm{EXTRA and DIGing algorithms are special cases of our algorithm for specific choices of matrices $A$ and $B$ and one primal update per iteration, and thus have close performance to PD - 1.} In the NEAR-DGD$^{+}$ the number of communications increases linearly with the iteration number, which results in slow convergence with respect to the number of communications. We obtained similar results for other standard machine learning datasets, including \textit{diabetes-scale}, \textit{heart-scale}, \textit{a1a}, \textit{australian-scale}, and \textit{german} \cite{chang2011libsvm}. 
\begin{figure}
\vskip 0.13in
  \centering
  \subfloat{\includegraphics[width=0.228\textwidth]{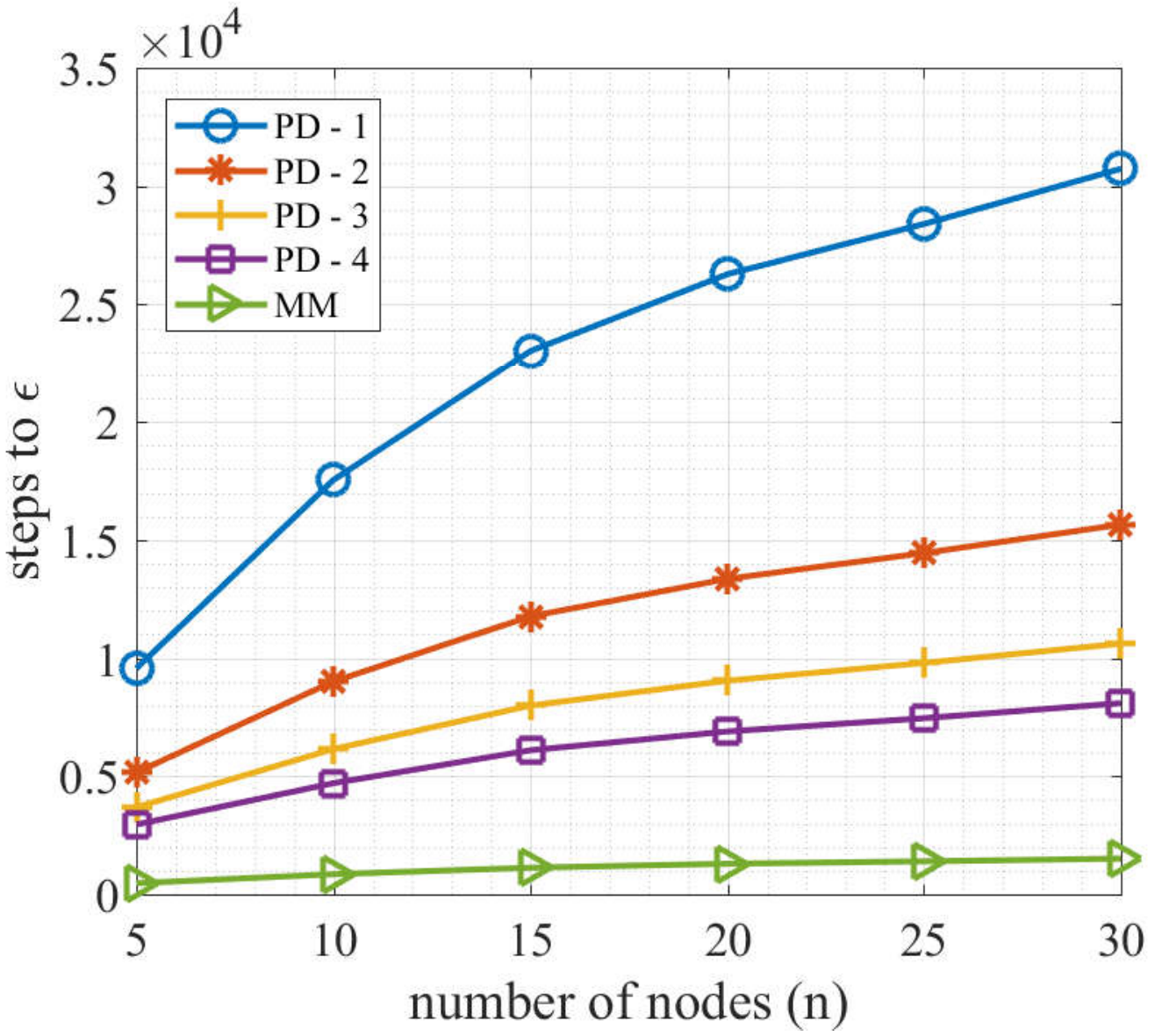}}
  \hfill
  \subfloat{\includegraphics[width=0.23\textwidth]{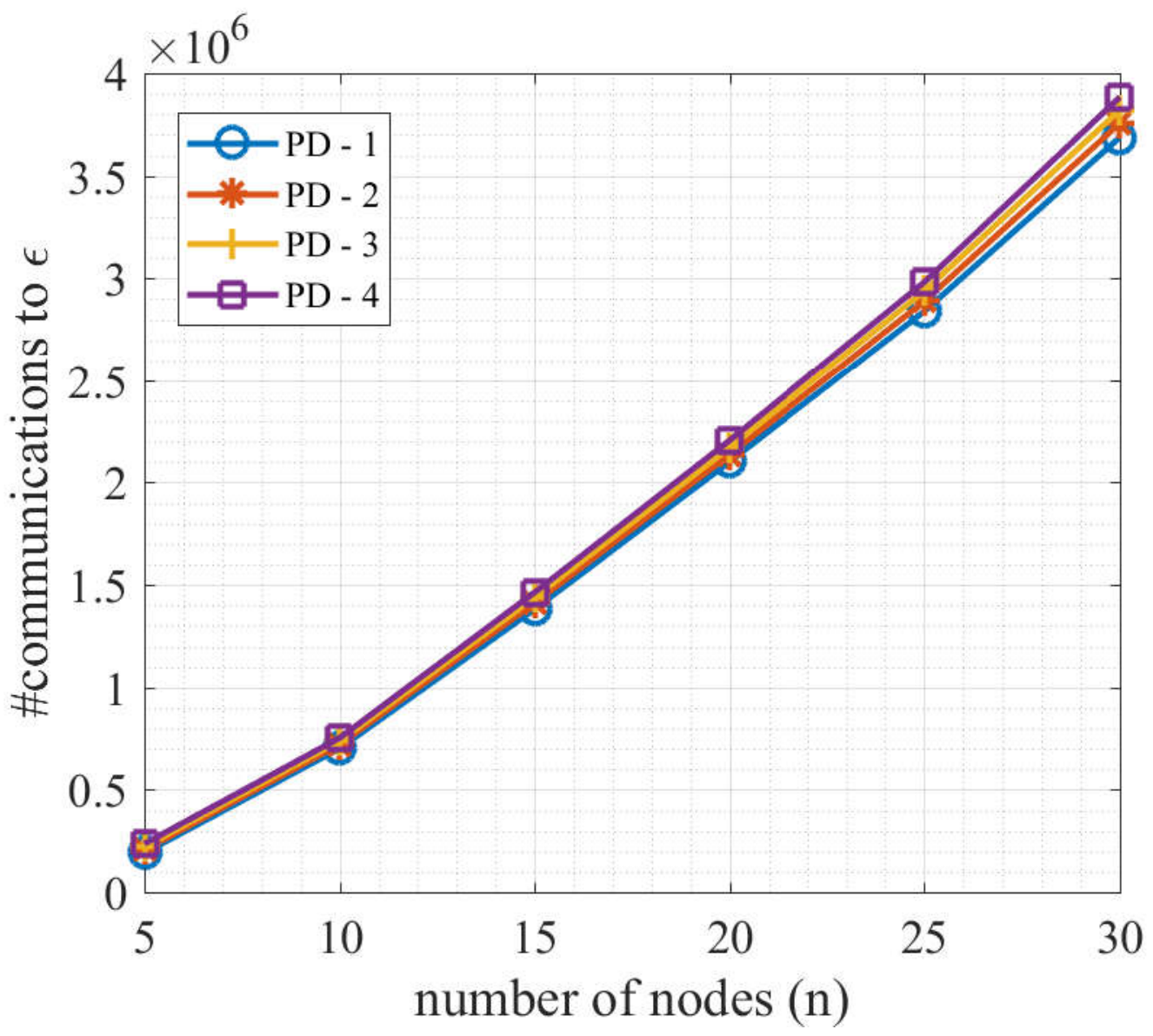}}
  \caption{Performance of primal-dual algorithm for $T=1, ...,4$ in terms of average number of steps and communications until the relative error is less than $\epsilon$.} \label{fig2}
  \vskip -0.2in
\end{figure}
\section{Final Remarks}
\fm{In this paper we propose a general framework of first order decentralized primal-dual optimization algorithms.
Our general class of algorithms allows for multiple primal updates per iteration, which results in design flexibility to control the trade off between execution complexity and performance of the algorithm. We show that the proposed class of algorithms converges to the exact solution with a global linear rate. The numerical experiments show the convergence speed improvement of primal-dual algorithms with multiple primal updates per iteration compared to other \ew{known} first order methods like EXTRA, DIGing, and NEAR-DGD$^+$. Possible future work includes analysis of the convergence properties for non-convex objective functions and extending the framework to second order primal-dual algorithms.} 
\bibliographystyle{plain}
\bibliography{CDCRef}

\begin{thebibliography}{10}

\bibitem{berahas2018balancing}
Albert Berahas, Raghu Bollapragada, Nitish~Shirish Keskar, and Ermin Wei.
\newblock Balancing communication and computation in distributed optimization.
\newblock {\em IEEE Transactions on Automatic Control}, 2018.

\bibitem{bertsekas2011incremental}
Dimitri~P Bertsekas.
\newblock Incremental gradient, subgradient, and proximal methods for convex
  optimization: A survey.
\newblock {\em Optimization for Machine Learning}, 2010(1-38):3, 2011.

\bibitem{bertsekas2014constrained}
Dimitri~P Bertsekas.
\newblock {\em Constrained optimization and Lagrange multiplier methods}.
\newblock Academic press, 2014.

\bibitem{bertsimas1997introduction}
Dimitris Bertsimas and John~N Tsitsiklis.
\newblock {\em Introduction to linear optimization}, volume~6.
\newblock Athena Scientific Belmont, MA, 1997.

\bibitem{boyd2011distributed}
Stephen Boyd, Neal Parikh, Eric Chu, Borja Peleato, Jonathan Eckstein, et~al.
\newblock Distributed optimization and statistical learning via the alternating
  direction method of multipliers.
\newblock {\em Foundations and Trends{\textregistered} in Machine learning},
  3(1):1--122, 2011.

\bibitem{chang2011libsvm}
Chih-Chung Chang and Chih-Jen Lin.
\newblock {LIBSVM}: a library for support vector machines.
\newblock {\em ACM transactions on intelligent systems and technology (TIST)},
  2(3):27, 2011.

\bibitem{chen2012fast}
Annie~I Chen and Asuman Ozdaglar.
\newblock A fast distributed proximal-gradient method.
\newblock In {\em Communication, Control, and Computing (Allerton), 2012 50th
  Annual Allerton Conference on}, pages 601--608. IEEE, 2012.

\bibitem{eisen2017decentralized}
Mark Eisen, Aryan Mokhtari, and Alejandro Ribeiro.
\newblock {Decentralized quasi-Newton methods}.
\newblock {\em IEEE Transactions on Signal Processing}, 65(10):2613--2628,
  2017.

\bibitem{hestenes1969multiplier}
Magnus~R Hestenes.
\newblock Multiplier and gradient methods.
\newblock {\em Journal of optimization theory and applications}, 4(5):303--320,
  1969.

\bibitem{iutzeler2016explicit}
Franck Iutzeler, Pascal Bianchi, Philippe Ciblat, and Walid Hachem.
\newblock Explicit convergence rate of a distributed alternating direction
  method of multipliers.
\newblock {\em IEEE Transactions on Automatic Control}, 61(4):892--904, 2016.

\bibitem{jakovetic2015linear}
Du{\v{s}}an Jakoveti{\'c}, Jos{\'e}~MF Moura, and Joao Xavier.
\newblock Linear convergence rate of a class of distributed augmented
  lagrangian algorithms.
\newblock {\em IEEE Transactions on Automatic Control}, 60(4):922--936, 2015.

\bibitem{jakovetic2014fast}
Du{\v{s}}an Jakoveti{\'c}, Joao Xavier, and Jos{\'e}~MF Moura.
\newblock Fast distributed gradient methods.
\newblock {\em IEEE Transactions on Automatic Control}, 59(5):1131--1146, 2014.

\bibitem{kekatos2013distributed}
Vassilis Kekatos and Georgios~B Giannakis.
\newblock Distributed robust power system state estimation.
\newblock {\em IEEE Transactions on Power Systems}, 28(2):1617--1626, 2013.

\bibitem{ling2015dlm}
Qing Ling, Wei Shi, Gang Wu, and Alejandro Ribeiro.
\newblock {DLM}: Decentralized linearized alternating direction method of
  multipliers.
\newblock {\em IEEE Trans. Signal Processing}, 63(15):4051--4064, 2015.

\bibitem{mansoori2017superlinearly}
Fatemeh Mansoori and Ermin Wei.
\newblock {Superlinearly convergent asynchronous distributed network Newton
  method}.
\newblock In {\em Decision and Control (CDC), 2017 IEEE 56th Annual Conference
  on}, pages 2874--2879. IEEE, 2017.

\bibitem{mokhtari2015network}
Aryan Mokhtari, Qing Ling, and Alejandro Ribeiro.
\newblock {Network Newton-part i: Algorithm and convergence}.
\newblock {\em arXiv preprint arXiv:1504.06017}, 2015.

\bibitem{mokhtari2016dsa}
Aryan Mokhtari and Alejandro Ribeiro.
\newblock {DSA}: Decentralized double stochastic averaging gradient algorithm.
\newblock {\em The Journal of Machine Learning Research}, 17(1):2165--2199,
  2016.

\bibitem{mokhtari2015decentralized}
Aryan Mokhtari, Wei Shi, Qing Ling, and Alejandro Ribeiro.
\newblock Decentralized quadratically approximated alternating direction method
  of multipliers.
\newblock In {\em Signal and Information Processing (GlobalSIP), 2015 IEEE
  Global Conference on}, pages 795--799. IEEE, 2015.

\bibitem{mokhtari2016decentralized}
Aryan Mokhtari, Wei Shi, Qing Ling, and Alejandro Ribeiro.
\newblock A decentralized second-order method with exact linear convergence
  rate for consensus optimization.
\newblock {\em IEEE Transactions on Signal and Information Processing over
  Networks}, 2(4):507--522, 2016.

\bibitem{mota2013d}
Joao~FC Mota, Joao~MF Xavier, Pedro~MQ Aguiar, and Markus Puschel.
\newblock {D-ADMM}: A communication-efficient distributed algorithm for
  separable optimization.
\newblock {\em IEEE Transactions on Signal Processing}, 61(10):2718--2723,
  2013.

\bibitem{nedic2011asynchronous}
Angelia Nedi{\'c}.
\newblock Asynchronous broadcast-based convex optimization over a network.
\newblock {\em IEEE Transactions on Automatic Control}, 56(6):1337--1351, 2011.

\bibitem{nedic2001convergence}
Angelia Nedi{\'c} and Dimitri Bertsekas.
\newblock Convergence rate of incremental subgradient algorithms.
\newblock In {\em Stochastic optimization: algorithms and applications}, pages
  223--264. Springer, 2001.

\bibitem{nedic2017achieving}
Angelia Nedi{\'c}, Alex Olshevsky, and Wei Shi.
\newblock Achieving geometric convergence for distributed optimization over
  time-varying graphs.
\newblock {\em SIAM Journal on Optimization}, 27(4):2597--2633, 2017.

\bibitem{nedic2018improved}
Angelia Nedi{\'c}, Alex Olshevsky, and Wei Shi.
\newblock Improved convergence rates for distributed resource allocation.
\newblock In {\em 2018 IEEE Conference on Decision and Control (CDC)}, pages
  172--177. IEEE, 2018.

\bibitem{nedic2009distributed}
Angelia Nedi{\'c} and Asuman Ozdaglar.
\newblock Distributed subgradient methods for multi-agent optimization.
\newblock {\em IEEE Transactions on Automatic Control}, 54(1):48--61, 2009.

\bibitem{predd2007distributed}
Joel~B Predd, Sanjeev~R Kulkarni, and H~Vincent Poor.
\newblock {\em Distributed learning in wireless sensor networks}.
\newblock John Wiley \& Sons: Chichester, UK, 2007.

\bibitem{ram2009incremental}
S~Sundhar Ram, A~Nedi{\'c}, and Venugopal~V Veeravalli.
\newblock Incremental stochastic subgradient algorithms for convex
  optimization.
\newblock {\em SIAM Journal on Optimization}, 20(2):691--717, 2009.

\bibitem{ram2010distributed}
S~Sundhar Ram, Angelia Nedi{\'c}, and Venugopal~V Veeravalli.
\newblock Distributed stochastic subgradient projection algorithms for convex
  optimization.
\newblock {\em Journal of optimization theory and applications},
  147(3):516--545, 2010.

\bibitem{ren2007information}
Wei Ren, Randal~W Beard, and Ella~M Atkins.
\newblock Information consensus in multivehicle cooperative control.
\newblock {\em IEEE Control systems magazine}, 27(2):71--82, 2007.

\bibitem{schmidt2005minfunc}
Mark Schmidt.
\newblock minfunc: unconstrained differentiable multivariate optimization in
  matlab.
\newblock {\em Software available at http://www. cs. ubc. ca/\~{}
  schmidtm/Software/minFunc. htm}, 2005.

\bibitem{shi2015extra}
Wei Shi, Qing Ling, Gang Wu, and Wotao Yin.
\newblock {EXTRA}: An exact first-order algorithm for decentralized consensus
  optimization.
\newblock {\em SIAM Journal on Optimization}, 25(2):944--966, 2015.

\bibitem{shi2014linear}
Wei Shi, Qing Ling, Kun Yuan, Gang Wu, and Wotao Yin.
\newblock On the linear convergence of the {ADMM} in decentralized consensus
  optimization.
\newblock {\em IEEE Trans. Signal Processing}, 62(7):1750--1761.

\bibitem{tsianos2012consensus}
Konstantinos~I Tsianos, Sean Lawlor, and Michael~G Rabbat.
\newblock Consensus-based distributed optimization: Practical issues and
  applications in large-scale machine learning.
\newblock In {\em 2012 50th Annual Allerton Conference on Communication,
  Control, and Computing (Allerton)}, pages 1543--1550. IEEE, 2012.

\bibitem{tutunov2016distributed}
Rasul Tutunov, Haitham~Bou Ammar, and Ali Jadbabaie.
\newblock {A distributed Newton method for large scale consensus optimization}.
\newblock {\em arXiv preprint arXiv:1606.06593}, 2016.

\bibitem{wei2012distributed}
Ermin Wei and Asuman Ozdaglar.
\newblock Distributed alternating direction method of multipliers.
\newblock 2012.

\bibitem{wei20131}
Ermin Wei and Asuman Ozdaglar.
\newblock On the {O}(1/k) convergence of asynchronous distributed alternating
  direction method of multipliers.
\newblock In {\em Global conference on signal and information processing
  (GlobalSIP), 2013 IEEE}, pages 551--554. IEEE, 2013.

\bibitem{wu2018decentralized}
Tianyu Wu, Kun Yuan, Qing Ling, Wotao Yin, and Ali~H Sayed.
\newblock Decentralized consensus optimization with asynchrony and delays.
\newblock {\em IEEE Transactions on Signal and Information Processing over
  Networks}, 4(2):293--307, 2018.

\bibitem{xiao2014proximal}
Lin Xiao and Tong Zhang.
\newblock A proximal stochastic gradient method with progressive variance
  reduction.
\newblock {\em SIAM Journal on Optimization}, 24(4):2057--2075, 2014.

\end{thebibliography}
\end{document}